\documentclass[pdflatex,sn-mathphys-num]{sn-jnl}% Math and Physical Sciences Numbered Reference Style

\usepackage{graphicx}%
\usepackage{multirow}%
\usepackage{amsmath,amssymb,amsfonts}%
\usepackage{amsthm}%
\usepackage{mathrsfs}%
\usepackage[title]{appendix}%
\usepackage{xcolor}%
\usepackage{textcomp}%
\usepackage{manyfoot}%
\usepackage{booktabs}%
\usepackage{algorithm}%
\usepackage{algorithmicx}%
\usepackage{algpseudocode}%
\usepackage{listings}%

\theoremstyle{thmstyleone}%
\newtheorem{theorem}{Theorem}%  meant for continuous numbers  lemma
\newtheorem{lemma}{Lemma}

\theoremstyle{thmstyletwo}%
\newtheorem{example}{Example}%

\theoremstyle{thmstylethree}%

\begin{document}

\title[On damping a delay control system with global contraction on a temporal graph]{On damping a delay control system with global contraction on a temporal tree}

%%=============================================================%%
%% GivenName	-> \fnm{Joergen W.}
%% Particle	-> \spfx{van der} -> surname prefix
%% FamilyName	-> \sur{Ploeg}
%% Suffix	-> \sfx{IV}
%% \author*[1,2]{\fnm{Joergen W.} \spfx{van der} \sur{Ploeg} 
%%  \sfx{IV}}\email{iauthor@gmail.com}
%%=============================================================%%

\author*[1,2,3]{\fnm{Aleksandr} \sur{Lednov}}\email{lednovalexsandr@gmail.com}

\affil*[1]{\orgname{Saratov State University}, \orgaddress{\street{Astrakhanskaya st., 83}, \city{Saratov}, \postcode{410012}, \country{Russia}}} 

\affil[2]{\orgname{Moscow Center for Fundamental and Applied Mathematics}, \city{Moscow}, \country{Russia}}

\affil[3] {\orgname{Lomonosov Moscow State University}, \orgaddress{\street{Leninskiye Gory, 1}, \city{Moscow}, \postcode{119991}, \country{Russia}}} %, 

\abstract{
We consider the problem of damping a control system with delay described by first-order functional-differential equations on a temporal tree. 
The delay in the system is time-proportional and propagates through the internal vertices. 
The problem of minimizing the energy functional with account of the probabilities of the scenarios corresponding to different edges is studied.
We establish the equivalence of this variational problem to a certain boundary value problem for second-order functional-diffferential equations on the tree, possessing both the global contractions and the global extensions, and prove the unique solvability of both problems. 
In particular, it is established that the optimal trajectory obeys Kirchhoff-type conditions at the internal vertices.
}



\maketitle

\section{Introduction}\label{sec1}

Krasovskii \cite{bib1} has posed and studied the problem of damping a control system with aftereffect described by an equation of retarded type with constant delay. Later on, Skubachevskii \cite{bib2,bib3} considered a generalization of this problem to the case in which the equation contains the delay also in the dominant terms, that is, belongs to neutral type (see also \cite{bib4} and the references therein). Recently, Buterin extended this problem to the so-called temporal graphs \cite{bib5,bib6,bib7,bib8}.

Differential operators on graphs, often called quantum graphs, have been extensively studied since the previous century in connection with modeling various processes in complex systems that can be represented as spatial networks \cite{bib9,bib10,bib11,bib12,bib13}. Such models typically involve Kirchhoff-type conditions at the internal vertices. Studies \cite{bib5,bib6,bib7,bib8} have shown that analogous conditions are satisfied by optimal trajectories on temporal graphs.

The concept of functional-differential operators on graphs with global delay was suggested in \cite{bib14}. In other words, the delay propagates through all internal vertices of the graph to the subsequent edges.  Specifically, the solution on the incoming edge determines the initial function for the equations on the outgoing edges. The global delay has become an alternative to the locally nonlocal settings introduced by Wang and Yang in \cite{bib15}, where the equation on each edge possessed its own delay parameter and could be considered independently.

The introduction of global delay made it possible to extend to graphs \cite{bib5,bib6} the mentioned problem of damping a control system with aftereffect. This has led to the concept of a temporal graph, the edges of which are parameterised by the time variable, while at each internal vertex, several distinct scenarios arise for the further course of the process in accordance with the number of edges emanating from that vertex. In \cite{bib5,bib7}, a stochastic interpretation of control systems on temporal trees was proposed. For example, such a system can be obtained by replacing the coefficients in the equation on an interval with discrete-time stochastic processes having a finite number of states in $\mathbb{R}$. The countable number of states corresponds to an infinite temporal tree \cite{bib7}. In \cite{bib8}, a globally nonlocal integro-differential control system on a star-graph was studied.

In the present paper, we extend to graphs the problem of damping a control system described by the so-called pantograph equation, which has many important applications \cite{bib16,bib17,bib18,bib19,bib20,bib21}. 
In this case, the delay is not constant but it is a time-proportional contraction. 
Recently, differential equations with this type of delay have also attracted special attention in the inverse spectral theory \cite{bib26}.

The case of the damping problem for a control system with time-proportional delay on an interval was considered by Rossovskii \cite{bib22} for the neutral equation:
\begin{equation}
y^{\prime}\left(t\right)+a y^{\prime}\left(q^{-1} t\right)+b y\left(t\right)+c y\left(q^{-1} t\right)=u\left(t\right), \quad t>0, 
\label{eq_i1}
\end{equation}
where $a$, $b$, $c \in \mathbb{R}$ and $q>1$, while $u\left(t\right)$ is a real-valued square-integrable control function. The state of the system at the initial time is defined by the condition
\begin{equation}
y\left(0\right)=y_0 \in \mathbb{R}. 
\label{eq_i2}
\end{equation}

The control problem can be formulated as follows. Find  such a control $u\left(t\right)$ that would bring the system (\ref{eq_i1}) and (\ref{eq_i2}) into the equilibrium $y\left(t\right)=0$ for $t\geq T$.

For achieving this, it is sufficient to find $u\left(t\right)\in L_2\left(0,T\right)$ leading to the state
\begin{equation}
y\left(t\right)=0, \quad q^{-1}T \leqslant t \leqslant T, 
\label{eq_i3}
\end{equation}
and to put $u\left(t\right) = 0$ for $t\geq T$. Since such $u\left(t\right)$ is not unique, it is reasonable to look for it trying to minimize the efforts $\|u\|_{L_2\left(0,T\right)}$.

This leads to the variational problem of minimizing the egergy functional
\begin{equation}
\mathcal{J}\left(y\right)=\displaystyle\int_0^{T}\left(y^{\prime}\left(t\right)+a y^{\prime}\left(q^{-1} t\right) +b y\left(t\right)+c y\left(q^{-1} t\right)\right)^2 d t \longrightarrow \min 
\label{eq_i4}
\end{equation}
over the set of functions $y\left(t\right)\in W^1_2[0,T]$ satisfying the boundary conditions (\ref{eq_i2}), (\ref{eq_i3}).

Solution of the problem (\ref{eq_i2})-(\ref{eq_i4}) was obtained in \cite{bib22}, where the problem was reduced to an equivalent boundary value problem for a second-order functional-differential equation. In particular, it was shown that if the function $y(t) \in W_2^1[0, T]$ satisfies conditions (\ref{eq_i2}) and (\ref{eq_i3}) and minimizes the functional (\ref{eq_i4}), then the integral identity 
\begin{multline*}
\int_0^{q^{-1}T}\left(\left(1+a^2 q\right) y^{\prime}(t)+a y^{\prime}\left(q^{-1} t\right)+a q y^{\prime}(q t)\right) v^{\prime}(t) d t \\
+\int_0^{q^{-1}T}\left(\left(a b-c q^{-1}\right) y^{\prime}\left(q^{-1} t\right)+\left(c q-a b q^2\right) y^{\prime}(q t)\right.\\
\left.+\left(b^2+c^2 q\right) y(t)+b c y\left(q^{-1} t\right)+b c q y(q t)\right) v(t) d t=0
\end{multline*}
holds for all $v \in W_2^1[0, T]$ such that $v\left(0\right)=0$ and $v\left(t\right)=0$ on $\left[q^{-1}T,T\right]$.

This means that the function $y(t) \in W_2^1[0, T]$ is a solution of the boundary value problem for the equation
\begin{multline}
-\left(\left(1+a^2 q\right) y^{\prime}(t)+a y^{\prime}\left(q^{-1} t\right)+a q y^{\prime}(q t)\right)^{\prime}
\\
+\left(a b-c q^{-1}\right) y^{\prime}\left(q^{-1} t\right)+\left(c q-a b q^2\right) y^{\prime}(q t)
\\
+\left(b^2+c^2 q\right) y(t)+b c y\left(q^{-1} t\right)+b c q y(q t)=0, \quad 0<t<q^{-1}T, 
\label{eq_i5}
\end{multline}
under conditions (\ref{eq_i2}) and (\ref{eq_i3}). Since this boundary value problem may have no solution in $W_2^2\left[0,q^{-1}T\right]$,  the solution is understood in the generalized sense, that is,
$$
\left(1+a^2 q\right) y^{\prime}(t)+a y^{\prime}\left(q^{-1} t\right)+a q y^{\prime}(q t)\in W_2^1\left[0,q^{-1}T\right].
$$

The converse is also true: if $y(t) \in W_2^1\left[0, T\right]$ is a generalized solution of the problem (\ref{eq_i2}), (\ref{eq_i3}), (\ref{eq_i5}), then $y$ minimizes the functional (\ref{eq_i4}).

The following theorem gives the existence and uniqueness of a generalized solution to the boundary value problem (\ref{eq_i2})--(\ref{eq_i4}), which also means the unique solvability of the variational problem (\ref{eq_i2})-(\ref{eq_i4}).

\begin{theorem}[\cite{bib22}]\label{T1}
Let $|a| \neq q^{-1/2}$. Then the boundary value problem (\ref{eq_i2})--(\ref{eq_i4}) has a unique generalized solution $y \in W_2^1[0, T]$.
\end{theorem}

In the present work, we generalize this problem to an arbitrary tree. For simplicity, we restrict our attention to the case $a=0$, i.e. we consider a retarded-type control system.
The netral system on a star-graph was briefly addressed in \cite{bib24}. The case of an arbitrary tree requires a separate investigation.

To illustrate the specifics of the problem on graphs, in the next section, we formulate the obtained results on a star-shaped graph (see also  \cite{bib23}). 
In Section 3, we introduce the control system with time-proportional delay on an arbitrary tree and formulate the corresponding variational problem. In Section 4, we establish its equivalence to a boundary value problem on the tree. In the last section, the unique solvability of both problems is proved.

\section{Star-shaped graph}\label{sec2}

Let up to the time point $t=T_1$, associated with the internal vertex $v_1$ of the graph $\Gamma_m$ in Figure \ref{fig1}, our control system with time-proportional delay be described by the equation
\begin{equation}
\ell_1 y\left(t\right):=y_1^{\prime}\left(t\right)+b_1 y_1\left(t\right)+c_1 y_1\left(q^{-1}t\right)=u_1\left(t\right),\quad 0<t<T_1,
\label{eq_s1}
\end{equation}
where $y_1(t)$ is defined on the edge $e_1=\left[v_0, v_1\right]$ of $\Gamma_m$, with the initial condition
\begin{equation}
y_1\left(0\right)=y_0.
\label{eq_s2}
\end{equation}
\begin{figure}[h]
\centering
\includegraphics[width=0.9\textwidth]{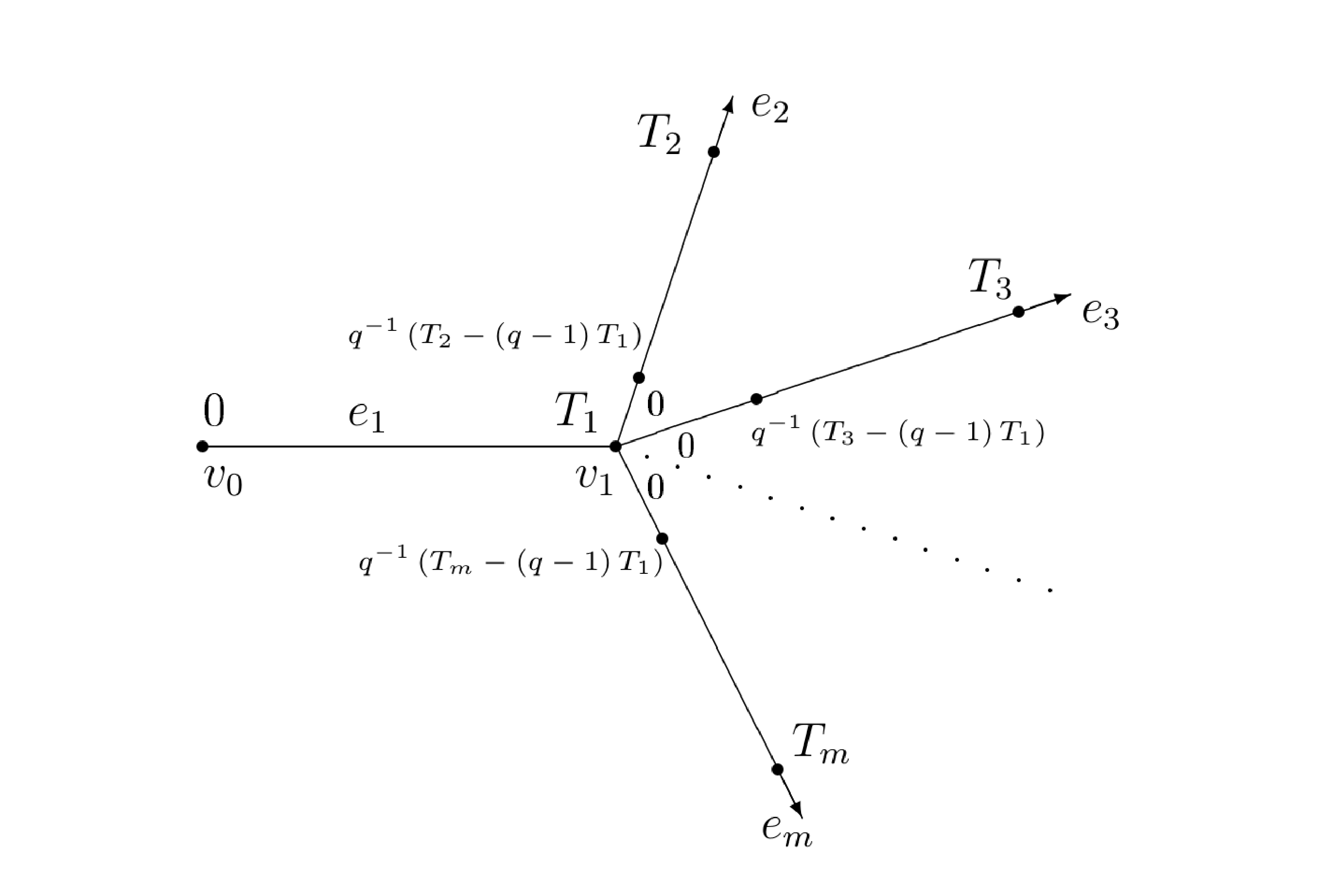}
\caption{A star-shaped graph $\Gamma_m$}\label{fig1}
\end{figure}

At the vertex $v_1$, this system branches into $m-1$ independent processes described by the equations
\begin{equation}
\ell_j y\left(t\right):=y_j^{\prime}\left(t\right)+b_j y_j\left(t\right)+c_j y_j\left(q^{-1}\left(t - \left(q-1\right)T_1\right)\right)=u_j\left(t\right),\;\;\; t>0,\;\;\; j=\overline{2, m},
\label{eq_s3}
\end{equation}
but having a common history determined by equation (\ref{eq_s1}) along with initial conditions (\ref{eq_s2}) and
\begin{equation}
y_j\left(t\right)=y_1\left(t+T_1\right), \quad \left(q^{-1}-1\right)T_1 < t < 0, \quad j=\overline{2,m},
\label{eq_s4}
\end{equation}
as well as the continuity conditions at $v_1$ :
\begin{equation}
y_j\left(0\right)=y_1\left(T_1\right), \quad j=\overline{2, m},
\label{eq_s5}
\end{equation}
agreeing with (\ref{eq_s4}) as $t \to 0^-$.

As in the preceding section, we assume that $q>1$, $y_0\in\mathbb{R}$ and $b_j$, $c_j\in\mathbb{R}$, $j=\overline{1, m}$.

For $j=\overline{2, m}$, the $j$th equation in (\ref{eq_s3}) is defined on the infinite edge $e_j$ of $\Gamma_m$. Conditions (\ref{eq_s4}) mean that the delay propagates through $v_1$.

The problem (\ref{eq_s1})-(\ref{eq_s5}) has a unique solution $y_j\left(t\right)\in W_2^1\left[0,T_j\right]$, $j=\overline{1, m}$, for arbitrary fixed $T_j>0$, $j=\overline{2, m}$, whenever $u_j(t) \in L_2(0, T_j)$ for $j = \overline{1,m}$. Indeed, on the edge $e_1$, it can be reduced to solving a Volterra equation of the second kind with a piecewise-constant kernel and a free term from $W_2^1[0,T_1]$:
$$
y_1(t)=f(t)+\displaystyle\int_0^{t}K(t,s)y_1(s)ds,
$$
where $$f(t)=y_0+\displaystyle\int_0^{t} u_1(s)ds,$$
$$K(t, s) = 
\begin{cases}
- (b_1 +  qc_1), &  s \leq q^{-1}t, \\
- b_1, & s > q^{-1}t.
\end{cases}
$$ 
Then, the obtained $y_1(t)$ gives the common initial function (\ref{eq_s4}) and the common initial condition (\ref{eq_s5}) for all equations in (\ref{eq_s3}), which can be solved similarly.

For definiteness, let $T_j>\left(q-1\right)T_1$ for all $j=\overline{2, m}$. Analogously to the preceding section, one needs to
bring the system (\ref{eq_s1})-(\ref{eq_s5}) into the equilibrium state
\begin{equation}
y_j\left(t\right)=0, \quad q^{-1}\left(T_j-\left(q-1\right)T_1\right)\leqslant t\leqslant T_j, \quad j=\overline{2,m},
\label{eq_s6}
\end{equation}
by choosing suitable controls $u_j(t) \in L_2\left(0, T_j\right), j=\overline{1, m}$. Then, letting $u_j(t)=0$ for $t>T_j$ and $j=\overline{2, m}$, we will have $y_j(t)=0$ for such $t$ and $j$. In other words, the system will be damped on each outgoing edge. Since such $u_j(t)$ are not unique, it is natural to try reducing the efforts $\left\|u_j\right\|_{L_2\left(0, T_j\right)}$ as much as possible. In order to regulate the participation of each $\left\|u_j\right\|_{L_2\left(0, T_j\right)}^2$ in the corresponding energy functional by choosing a certain positive weight $\alpha_j$.

Thus, we arrive at the variational problem
\begin{equation}
\sum_{j=1}^m \alpha_j \int_0^{T_j}\left(\ell_j y(t)\right)^2 d t \rightarrow \min
\label{eq_s7}
\end{equation}
under conditions (\ref{eq_s2}) and (\ref{eq_s4})-(\ref{eq_s6}), where $\alpha_j>0$, $j=\overline{1, m}$, are fixed.

To choose the weights $\alpha_j$, $j = \overline{1, m}$, one can employ a probabilistic approach~\cite{bib5}. In this case, it is assumed that $\alpha_1 = 1$, and the values of $\alpha_j$ for $j = \overline{2, m}$ are determined as the probabilities of the scenarios defined by the equations in (\ref{eq_s3}), so that $$\alpha_2 + \ldots + \alpha_m = 1.$$ 
Coincidence of all the scenarious would lead to the interval case (see example 2 in \cite{bib5}).

The following theorems hold.

\begin{theorem}\label{T2}
Functions $y_j\in W_2^1\left[0,T_j\right]$, $j=\overline{1, m}$, form a solution of the variational problem (\ref{eq_s2}), (\ref{eq_s4})-(\ref{eq_s7}) if and only if they possess the additional smoothness
$$
y_1\left(t\right)\in W_2^2\left[0,T_1\right],\quad y_j\left(t\right)\in W_2^2\left[0,q^{-1}\left(T_j-\left(q-1\right)T_1\right)\right], \quad j=\overline{2, m},
$$
and solve the boundary value problem (which we denote by $\mathcal{B}$) consisting of the equations 
$$
\alpha_1\left(\ell_jy\right)^{\prime}\left(t\right)= \alpha_1b_j \ell_j y\left(t\right) +
\begin{cases}
q \alpha_1c_1\ell_1 y\left(qt\right), & 0<t<q^{-1}T_1,
\\
q \displaystyle\sum_{k=2}^m \alpha_k c_k\ell_k y\left(qt-T_1\right), & q^{-1}T_1< t<T_1, 
\end{cases}
$$
$$
\left(\ell_jy\right)^{\prime}\left(t\right)= b_j \ell_j y\left(t\right) + q  c_j\ell_j y\left(qt+\left(q-1\right)T_1\right),\quad 0<t<q^{-1}\left(T_j-\left(q-1\right)T_1\right), \quad j=\overline{2, m},
$$
along with the standing conditions (\ref{eq_s2}) and (\ref{eq_s4})-(\ref{eq_s6}) as well as the Kirchhoff-type condition
$$
\alpha_1 y_1^\prime\left(T_1\right)+
\Bigg( 
\alpha_1 b_1-\displaystyle\sum_{j=2}^m \alpha_j b_j
\Bigg)
y_1\left(T_1\right)+
\Bigg(
\alpha_1c_1-\displaystyle\sum_{j=2}^m \alpha_j c_j
\Bigg)
y_1\left(q^{-1}T_1\right)=\displaystyle\sum_{j=2}^m \alpha_jy_j^\prime\left(0\right),
$$
additionally emerging at the internal vertex $v_1$.
\end{theorem}

\begin{theorem}\label{T3}
The boundary value problem $\mathcal{B}$ has a unique solution. Moreover, this solution satisfies the estimate
$$
\|y_1\|_{W_2^1\left[0,T_1\right]}+\sum_{j=2}^m\|y_j\|_{W_2^1\left[0, q^{-1}\left(T_j-\left(q-1\right)T_1\right)\right]} \leq C|y_0|,
$$
where $C$ is independent of $y_0$.
\end{theorem}

These theorems establish the existence and uniqueness of an optimal trajectory $\left[y_1, y_2,\dots,y_m\right]$ with account of all possible scenarios $e_2,\dots,e_m$ simultaneously. Substituting $\left[y_1, y_2,\dots,y_m\right]$ into the equations (\ref{eq_s1}) and (\ref{eq_s3}), one can obtain the corresponding optimal control $\left[u_1, u_2,\dots,u_m\right]$.

Below, we generalize Theorems \ref{T2} and \ref{T3} to the case of an arbitrary tree (Theorems \ref{T4} and \ref{T5}). 
While the outgoing boundary edges of the tree are associated with temporal rays, we actually deal with a compact tree  $\mathcal{T}$ obtained by trimming the infinite edges at certain points $T_j$.

\section{Statement of the variation problem on a tree}\label{sec3}

Consider a compact tree $\mathcal{T}$ with the set of edges $\{e_1, \ldots, e_m\}$, internal vertices $\{v_1, \ldots, v_d\}$, and boundary vertices $\{v_0, v_{d+1}, \ldots, v_m\}$.
The vertex $v_0$ is labeled as the root.
The tree settings adopted in this work follow the structure described in \cite{bib5}. Under these settings, each edge $e_j$, $j = \overline{1,m}$, emanates from the corresponding vertex $v_{k_j}$ and terminates at $v_j$, where $k_1 = 0$. 
For each $j = \overline{0, d}$, the set $\left\{e_\nu\right\}_{\nu \in V_j}$, where
\begin{equation}
V_j:=\left\{\nu: k_\nu=j\right\}, 
\label{eq_t1}
\end{equation}
coincides with the set of edges emanating from the vertex $v_j$. 
The unique simple path between a vertex $v_j$ and the root is given by the chain of edges $\{e_{k_j^{< \nu >}}\}_{\nu = \overline{0,\nu_j}}$, where $k_j^{<0>} := j$ and $k_j^{<\nu+1>} := k_{k_j^{<\nu>}}$ for $\nu=\overline{0, \nu_j}$; the index $\nu_j$ is defined by $k_j^{<\nu_j>} = 1$.

Each edge $e_j$ is parametrized by variable $t \in\left[0, T_j\right]$, where $t=0$ corresponds to its initial vertex $v_{k_j}$ and $t=T_j$ to its terminal vertex $v_j$.
By a function $y$ on $\mathcal{T}$, we mean an $m$-tuple $y=\left[y_1, \ldots, y_m\right]$ whose component $y_j$ is defined on the edge $e_j$, that is, $y_j=y_j\left(t\right)$, $t \in\left[0, T_j\right]$.

Throughout the rest of the paper, we use the notation
$$
q_j\left(t\right) := q^{-1}\left(t - (q - 1)\tilde{T}_j\right), \quad j =\overline{1, m},
$$
where $\tilde{T}_j := \sum_{\nu=1}^{\nu_j} T_{k_j^{<\nu>}}$ and $q > 1$. In particular, we have $\tilde{T}_1 = 0$ and $q_1(t)=q^{-1}t$.

For definiteness, let 
\begin{equation}
T_j>\left(q-1\right)\tilde{T}_j, \quad j=\overline{2, m}.
\label{eq_t-1}
\end{equation}
Consider the following Cauchy problem on $\mathcal{T}$: 
\begin{equation}
\ell_j y\left(t\right):=y_j^{\prime}\left(t\right)+b_j y_j\left(t\right)+c_j y_j\left(q_j\left(t\right)\right)=u_j\left(t\right),\quad 0<t<T_j,\quad j=\overline{1, m},
\label{eq_t2}
\end{equation}
\begin{equation} 
y_j\left(t\right)=y_{k_j}\left(t+T_{k_j}\right),\quad t \in\left(q_j\left(0\right), 0\right),\quad j=\overline{2, m}, 
\label{eq_t3}
\end{equation}
\begin{equation}
y_j\left(0\right)=y_{k_j}\left(T_{k_j}\right),\quad j=\overline{2, m}, 
\label{eq_t4}
\end{equation}
\begin{equation}
y_1\left(0\right)=y_0\in \mathbb{R}, 
\label{eq_t5}
\end{equation}
where  $b_j, c_j \in \mathbb{R}$ and $u_j \in L_2\left(0, T_j\right)$ for $j=\overline{1, m}$. While the $j$th equation in (\ref{eq_t2}) is defined on the edge $e_j$ of $\mathcal{T}$, relations (\ref{eq_t4}) become matching conditions at the internal vertices. Relations (\ref{eq_t3}) are initial-function conditions for all equations in (\ref{eq_t2}) except the first one. 
These conditions mean that the delay propagates through all internal vertices. Condition (\ref{eq_t5}) defines the initial state of the system for the entire tree.

Since $\tilde{T}_j = \tilde{T}_{k_j} + T_{k_j}$, we can observe that
$$
q_j(0) + T_{k_j} = \left(q^{-1}-1\right)\left(\tilde{T}_{k_j}+T_{k_j}\right)+T_{k_j} = q_{k_j}\left(T_{k_j}\right), \quad k_j=\overline{1, d}.
$$
This, together with  (\ref{eq_t-1}), shows that the argument $t+T_{k_j}$ in the right-hand side of (\ref{eq_t3}) is positive. 

We note that the problem (\ref{eq_s1})-(\ref{eq_s5}), obviously, coincides with the problem (\ref{eq_t2})-(\ref{eq_t5}) for $d = 1$. 
As in the preceding section, one can show that the latter problem also has a unique solution
$$
\left[y_1,\dots,y_m\right]\in\bigoplus_{j=1}^{m} W_2^1\left[0,T_j\right].
$$

\begin{example}
Let $m=d+1$. In this case, we have $\tilde{T}_j=\sum_{\nu=1}^{j-1} T_{\nu}$, $j=\overline{1, m}$.  Introduce the notations
$$
y(t) :=y_j\left(t-\tilde{T}_j\right),\quad u(t) :=u_j\left(t-\tilde{T}_j\right),\quad \tilde{T}_j\leq t\leq \tilde{T}_j+T_j, \quad j=\overline{1, m}.
$$
Additionally, suppose that  $b:=b_1=\dots=b_m$ and $c:=c_1=\dots=c_m$. Then the Cauchy problem (\ref{eq_t2})-(\ref{eq_t5}) takes the form (\ref{eq_i1}), (\ref{eq_i2}) with $a = 0$. 
\end{example}

Analogously to the preceding section, we intend to find a control function
$$u=\left[u_1, \ldots, u_m\right] \in L_2\left(\mathcal{T}\right):=\bigoplus_{j=1}^m L_2\left(0, T_j\right)$$
that leads to the equilibrium state
\begin{equation} 
y_j\left(t\right)=0, \quad t \in\left[q_j\left(T_j\right), T_j\right],\quad j=\overline{d+1, m} , 
\label{eq_t6}
\end{equation}
and minimizes all $\left\|u_j\right\|_{L_2\left(0, T_j\right)}^2$ with some positive weights $\alpha_j$, $j=\overline{1,m}$.

Thus, we arrive at the variational problem:
\begin{equation}
\mathcal{J}\left(y\right):=\sum_{j=1}^m \alpha_j \int_0^{T_j}\left(\ell_j y\left(t\right)\right)^2 d t \rightarrow \min 
\label{eq_t7}
\end{equation}
for the functions $y=\left[y_1, \ldots, y_m\right]$ defined on $\mathcal{T}$ under conditions (\ref{eq_t3})-(\ref{eq_t6}).

For brevity, we introduce the designation $\ell y:=\left[\ell_1 y, \ldots, \ell_m y\right]$ and agree that taking $\mathcal{J}\left(y\right)$ and $\ell y$ as well as $\ell_j y$ for $j=\overline{2, m}$ of any function $y$ on $\mathcal{T}$ automatically means the application of the initial-function conditions (\ref{eq_t3}).

For brevity, we introduce the designation $\ell y:=\left[\ell_1 y, \ldots, \ell_m y\right]$. We note that conditions (\ref{eq_t3}) do not restrict the function $y = [y_1, \ldots, y_m]$ and agree that taking  $\mathcal{J}\left(y\right)$ and $\ell y$ as well as $\ell_j y$ for $j=\overline{2, m}$ of any function $y$ on $\mathcal{T}$ automatically means the application of the initial-function conditions (\ref{eq_t3}).

\section{Reduction to a boundary value problem}\label{sec4}

Consider the real Hilbert space $W_2^k\left(\mathcal{T}\right):= \bigoplus_{j=1}^m W_2^k\left[0, T_j\right]$ with the natural inner product
$$
\left(y, z\right)_{W_2^k\left(\mathcal{T}\right)}:=\sum_{j=1}^m\left(y_j, z_j\right)_{W_2^k\left[0, T_j\right]},
$$
where $y=\left[y_1, \ldots, y_m\right]$ and $z=\left[z_1, \ldots, z_m\right]$, while 
$$
\left(f, g\right)_{W_2^k[a, b]}=\sum_{v=0}^k\left(f^{\left(v\right)}, g^{\left(v\right)}\right)_{L_2\left(a, b\right)}
$$ 
is the inner product in $W_2^k[a, b]$ and $\left(\cdot, \cdot\right)_{L_2\left(a, b\right)}$ is the one in $L_2\left(a, b\right)$. 

Denote by $\mathcal{W}$ the closed subspace of $W_2^1\left(\mathcal{T}\right)$ consisting of all $m$-tuples $\left[y_1, \ldots, y_m\right] \in W_2^1\left(\mathcal{T}\right)$ that obey the matching conditions (\ref{eq_t4}), the target conditions (\ref{eq_t6}), and $y_1(0)=0$.

We also introduce a subspace
$$W_2^k\left(\widetilde{\mathcal{T}}\right):=\bigoplus_{j=1}^d W_2^k\left[0, T_j\right] \oplus \bigoplus_{j=d+1}^m W_2^k\left[0, q_j\left(T_j\right)\right],$$
which differs from $W_2^k\left(\mathcal{T}\right)$ only by replacing $T_j$ with $q_j\left(T_j\right)$ for $j=\overline{d+1,m}$.

\begin{lemma}
If $y \in W_2^1\left(\mathcal{T}\right)$ is a solution of the variational problem (\ref{eq_t3})--(\ref{eq_t7}), then
\begin{equation}
B(y, w):=\sum_{j=1}^m \alpha_j \int_0^{T_j} \ell_j y(t) \ell_j w(t) d t=0 \quad \forall w \in \mathcal{W}. 
\label{eq_t8}
\end{equation}

Conversely, if $y \in W_2^1\left(\mathcal{T}\right)$ obeys (\ref{eq_t4})--(\ref{eq_t6}) and (\ref{eq_t8}), then $y$ is a solution of (\ref{eq_t3})--(\ref{eq_t7}).
\end{lemma}
\begin{proof}
Let $y \in W_2^1\left(\mathcal{T}\right)$ be a solution of (\ref{eq_t3})--(\ref{eq_t7}). Then for any $w \in \mathcal{W}$, the sum $y+s w$ belongs to $W_2^1\left(\mathcal{T}\right)$ whenever $s \in \mathbb{R}$, and obeys conditions (\ref{eq_t4})--(\ref{eq_t6}). Put
$$F(s):=\mathcal{J}(y+s w)=\mathcal{J}(y)+2 s B(y, w)+s^2 \mathcal{J}(w).$$
Since $\mathcal{J}(y+s w) \geq \mathcal{J}(y)$ for all $s \in \mathbb{R}$, we have $0=F^\prime(0)=2B(y,w)$.

Conversely, for any $y \in W_2^1\left(\mathcal{T}\right)$ obeying (\ref{eq_t4})--(\ref{eq_t6}), the fulfilment of (\ref{eq_t8}) implies 
$$\mathcal{J}(y+w)=\mathcal{J}(y)+2 B(y, w)+\mathcal{J}(w) \geq \mathcal{J}(y)$$
for all $w\in\mathcal{W}$, which gives (\ref{eq_t7}). 
\end{proof}

We transform relation (\ref{eq_t8}) by changing variables in the terms involving $w_j\left(q_j\left(t\right)\right)$. Then, in accordance with the definition in (\ref{eq_t2}), the expression for $B(y, w)$ takes the following form:
\begin{multline*}
B\left(y, w\right) = \displaystyle\sum_{j=1}^m\alpha_j\displaystyle\int_0^{T_j}\ell_jy\left(t\right) w_j^\prime\left(t\right)dt 
\\
+\displaystyle\sum_{j=1}^m \alpha_j b_j\displaystyle\int_0^{T_j}\ell_jy\left(t\right) w_j\left(t\right)dt+ q\sum_{j=1}^m \alpha_j c_j \int^{q_j\left(T_j\right)}_{q_j\left(0\right)}\ell_jy\left(q^{-1}_j(t)\right)w_j(t)dt,
\end{multline*}
where $q_j^{-1}(t) = qt + (q - 1)\tilde{T}_j$ is the inverse of $q_j(t)$.
Further, applying (\ref{eq_t3}) to $w~=~\left[w_1,\dots,w_m\right]\in\mathcal{W}$, one can represent
\begin{multline*}
\int^{q_j\left(T_j\right)}_{q_j\left(0\right)}\ell_jy\left(q^{-1}_j(t)\right)w_j(t)dt = \int^{q_j\left(T_j\right)}_{0}\ell_jy\left(q^{-1}_j(t)\right)w_j(t)dt
\\
+\int^{T_{k_j}}_{q_{k_j}\left(T_{k_j}\right)}\ell_jy\left(q^{-1}_j\left(t\right)-qT_{k_j}\right)w_{k_j}\left(t\right)dt, \quad j=\overline{1, m},
\end{multline*}
where $T_0=0$, $w_0=0$ and $q_0=0$. According to (\ref{eq_t1}), we have the summation rule
%\begin{equation}
$$
\sum_{j=2}^m A_j=\sum_{j=1}^d \sum_{\nu \in V_j} A_\nu \quad \left(\text { for any values } A_2, \ldots, A_m\right).
$$
%\label{eq_t9}
%\end{equation}

Hence, multiplying the preceding relation with $\alpha_j c_j$ and then summing up, we obtain
\begin{multline*}
\sum_{j=1}^m \alpha_j c_j \int^{
q_j\left(T_j\right)
}_{
q_j\left(0\right)
}\ell_jy\left(
q^{-1}_j\left(t\right)
\right)w_j(t)dt
=\sum_{j=1}^m \alpha_j c_j \int_0^{
q_j\left(T_j\right)
} \ell_j y\left(
q^{-1}_j\left(t\right)
\right) w_j(t) d t
\\
+\sum_{j=1}^d \sum_{\nu \in V_j} \alpha_\nu c_\nu \int_{
q_j\left(T_j\right)
}^{T_{j}}\ell_\nu y\left(
q^{-1}_\nu\left(t\right)-qT_j
\right) w_j(t) d t.
\end{multline*}

Thus one can rewrite (\ref{eq_t8}) in the equivalent form
\begin{multline}
B(y, w)=\sum_{j=1}^m\left(\alpha_j \int_0^{T_j} \ell_j y(t) w_j^{\prime}(t) d t\right.
\\
\left.+\int_0^{T_j}\left(\alpha_j b_j \ell_j y(t)+\tilde{\ell}_j y(t)\right) w_j(t) d t\right)=0, \quad w \in \mathcal{W},
\label{eq_t10}
\end{multline}
where
\begin{equation} 
\tilde{\ell}_j y\left(t\right) = 
\begin{cases}
q \alpha_jc_j\ell_j y\left( q_j^{-1}(t) \right), & 0<t<q_j\left(T_j\right), \;\;\;\quad j=\overline{1, m},
\\
& 
\\
q \displaystyle\sum_{\nu \in V_j} \alpha_\nu c_\nu\ell_\nu y\left(q_\nu^{-1}(t)-qT_j\right), & q_j\left(T_j\right)< t<T_j,\;\quad j=\overline{1, d}.
\end{cases} 
\label{eq_t11}
\end{equation}
while for $q_j\left(T_j\right)< t<T_j$, and $j=\overline{d+1, m}$ the expression $\tilde{\ell}_j y\left(t\right)$ can be defined as zero.

Denote by $\mathcal{B}$ the boundary value problem for the second-order functional--differential equations
\begin{equation}
\mathcal{L}_j y\left(t\right):=-\alpha_j\left(\ell_jy\right)^{\prime}\left(t\right)+\alpha_j b_j \ell_j y\left(t\right) + \widetilde{\ell_j}y\left(t\right)=0,\quad 0<t<l_j,\quad j=\overline{1,m},
\label{eq_t12}
\end{equation}
under conditions (\ref{eq_t3})--(\ref{eq_t6}) along with the Kirchhoff--type conditions
\begin{equation}
\alpha_jy^\prime_j\left(T_j\right)+\beta_jy_j\left(T_j\right)+\gamma_jy_j\left(q_j\left(T_j\right)\right)=\sum_{\nu \in V_j}\alpha_\nu y_\nu^\prime(0), \quad j=\overline{1, d},
\label{eq_t13}
\end{equation}
where the expressions $\tilde{\ell}_j y(t)$ are defined in (\ref{eq_t11}) and
\begin{equation}
l_j:=T_j,\quad j=\overline{1, d}, \quad l_j:= q_j\left(T_j\right), \quad j=\overline{d+1, m},
\label{eq_t14}
\end{equation}
\begin{equation}
\beta_j:=\alpha_j b_j-\sum_{\nu \in V_j} \alpha_\nu b_\nu, \quad \gamma_j:=\alpha_j c_j-\sum_{\nu \in V_j} \alpha_\nu c_\nu,\quad j=\overline{1, d}.
\label{eq_t15}
\end{equation}

The following lemma holds.

\begin{lemma}\label{L2}
If $y \in W_2^1\left(\mathcal{T}\right)$ obeys conditions (\ref{eq_t4})-(\ref{eq_t6}) and (\ref{eq_t8}), then $y \in W_2^2\left(\widetilde{\mathcal{T}}\right)$ and it solves the boundary value problem $\mathcal{B}$. Conversely, any solution y of $\mathcal{B}$ obeys (\ref{eq_t8}).
\end{lemma}
\begin{proof}
Taking into account that (\ref{eq_t8}) is equivalent to (\ref{eq_t10}) and applying Lemma 2 from \cite{bib5} to (\ref{eq_t10}) under the settings (\ref{eq_t14}), we get  $\ell_j y(t) \in W_2^1\left[0, l_j\right]$ for $j=\overline{1, m}$ and
\begin{equation}
\alpha_j \ell_j y\left(T_j\right)=\sum_{\nu \in V_j} \alpha_\nu \ell_\nu y(0),\quad j=\overline{1, d}.
\label{eq_t16}
\end{equation}

Recalling the definition of $\ell_j y$ in (\ref{eq_t2}) and using (\ref{eq_t3})-(\ref{eq_t4}) and (\ref{eq_t14}), we obtain the inclusion $y_j^{\prime}\left(t\right)=\ell_j y\left(t\right)-b_j y_j\left(t\right)-c_j y_j\left(q_j\left(t\right) \right) \in W_2^1\left[0, l_j\right]$ for $j=\overline{1, m}$, that is, $y \in W_2^2\left(\widetilde{\mathcal{T}}\right)$. Hence, we can rewrite (\ref{eq_t16}) in the equivalent form
\begin{multline}
y_j^{\prime}\left(T_j\right)+b_j y_j\left(T_j\right)+c_j y_j\left(q_j\left(T_j\right)\right)=
\\
=\frac{1}{\alpha_j} \sum_{\nu \in V_j} \alpha_\nu\left(y_\nu^{\prime}(0)+b_\nu y_\nu(0)+c_\nu y_\nu\left(q_\nu\left(0\right)\right)\right), \quad j=\overline{1, d},
\label{eq_t17}
\end{multline}
where, by virtue of (\ref{eq_t3}), the right-hand limits 
$$
y_\nu\left(q_\nu\left(0\right)\right):=\lim _{t \rightarrow\left(q_\nu\left(0\right)\right)^{+}} y_\nu(t)=\lim _{t \rightarrow\left(q_\nu\left(0\right) \right)^{+}}y_{k_\nu}\left(t+T_{k_\nu}\right)=y_{k_\nu}\left(q_{k_\nu}(T_{k_\nu})\right)=y_{j}\left(q_{j}(T_{j})\right),
$$
obviously, exist. Thus, relation (\ref{eq_t17}) along with (\ref{eq_t1}), (\ref{eq_t3}) and (\ref{eq_t4}) give (\ref{eq_t13}) with (\ref{eq_t15}).

Finally, integrating by parts in (\ref{eq_t10}) and using (\ref{eq_t4}), (\ref{eq_t6}) and the definition in (\ref{eq_t12}), we arrive at
\begin{equation}
B(y, w)=\sum_{j=1}^d w_j\left(l_j\right)\left(\alpha_j \ell_j y\left(l_j\right)-\sum_{\nu \in V_j} \alpha_\nu \ell_\nu y(0)\right)+\sum_{j=1}^m \int_0^{l_j} \mathcal{L}_j y(t) w_j(t) d t=0
\label{eq_t18}
\end{equation}
which, by virtue of (\ref{eq_t16}) and the variety of $w_j(t)$, gives (\ref{eq_t12}).

Conversely, let $y$ be a solution of the problem $\mathcal{B}$. Since (\ref{eq_t13}) is equivalent to (\ref{eq_t16}), the second equality in (\ref{eq_t18}) holds, which gives (\ref{eq_t8}).
\end{proof}

Combining Lemmas 1 and 2, we arrive at the main result of this section.

\begin{theorem}\label{T4}
A function $y\in W_2^1\left(\mathcal{T}\right)$ is a solution of the variational problem (\ref{eq_t3})--(\ref{eq_t7}) if and only if it belongs to $W_2^2\left(\widetilde{\mathcal{T}}\right)$ and solves the boundary value problem $\mathcal{B}$.
\end{theorem}

\section{The unique solvability}\label{sec5}

In this section, we establish the unique solvability of the boundary value problem $\mathcal{B}$ and thus, according to theorem \ref{T4}, of the variational problem (\ref{eq_t3})--(\ref{eq_t7}).

We begin with the following two auxiliary assertions.

\begin{lemma}\label{L3}
There exists $C_1$ such that
\begin{equation}
\|\ell w\|_{L_2(\mathcal{T})}^2 \leq C_1\|w\|_{W_2^1\left(\widetilde{\mathcal{T}}\right)}^2 \quad\forall w \in \mathcal{W}.
\label{eq_us1}
\end{equation}
\end{lemma}
\begin{proof}
Using the definition in (\ref{eq_t2}) and the inequality
\begin{equation}
\left(a_1+\ldots+a_n\right)^2 \leq n\left(a_1^2+\ldots+a_n^2\right),\quad a_1, \ldots, a_n \in \mathbb{R},
\label{eq_us2}
\end{equation}
for $n=3$, we obtain
\begin{equation}
\|\ell w\|_{L_2(\mathcal{T})}^2 \leq 3 \sum_{j=1}^m \int_0^{T_j}\left(w_j^{\prime}(t)\right)^2 d t+3 \sum_{j=1}^m b_j^2 \int_0^{T_j} w_j^2(t) d t+3 \sum_{j=1}^m c_j^2 \int_0^{T_j} w_j^2\left(q_j\left(t\right)\right)d t.
\label{eq_us3}
\end{equation}

Then, applying (\ref{eq_t3}) to $w$, we calculate
\begin{equation}
\int_0^{T_j} w_j^2\left(q_j\left(t\right)\right)d t= q\|w_j\|^2_{L_2\left(0, q_j\left(T_j\right)\right)}+q\|w_{k_j}\|^2_{L_2\left(q_{k_j}(T_{k_j}), T_{k_j}\right)}, \quad j=\overline{1, m},
\label{eq_us4}
\end{equation}
where $T_0=0$, $w_0=0$ and $q_0=0$. This identity along with (\ref{eq_us3}) gives (\ref{eq_us1}) with $C_1$ independent of $w$.
\end{proof}

\begin{lemma}\label{L4}
There exists $C_2>0$ such that
$$
\mathcal{J}(w) \geq C_2\|w\|_{W_2^1\left(\widetilde{\mathcal{T}}\right)}^2 \quad \forall w \in \mathcal{W}.
$$
\end{lemma}
\begin{proof}
To the contrary, let there exist $w_{(n)} \in \mathcal{W}, n \in \mathbb{N}$, such that
$$\mathcal{J}\left(w_{(n)}\right) \leq \frac{1}{n}\left\|w_{(n)}\right\|_{W_2^1\left(\widetilde{\mathcal{T}}\right)}^2.$$
Assuming without loss of generality that $\left\|w_{(n)}\right\|_{W_2^1\left(\widetilde{\mathcal{T}}\right)}=1$, we arrive at the inequalities
\begin{equation}
\mathcal{J}\left(w_{(n)}\right) \leq \frac{1}{n}, \quad n \in \mathbb{N}.
\label{eq_us6}
\end{equation}

Using the definition of $\ell_j y$ and inequality (\ref{eq_us2}) for $n=3$, we obtain
$$\left(w_j^{\prime}(t)\right)^2 \leq 3\left(\left(\ell_j w(t)\right)^2+b_j^2 w_j^2(t)+c_j^2 w_j^2\left(q_j\left(t\right)\right)\right), \quad j=\overline{1, m}.$$
Integrate this inequality from $0$ to $T_j$ and multiply with $\alpha_j$. Then, summing up with respect to $j$ and recalling the definition in (\ref{eq_t7}), we arrive at
\begin{equation}
\frac{1}{3} \sum_{j=1}^m \alpha_j \int_0^{T_j}\left(w_j^{\prime}(t)\right)^2 d t \leq \mathcal{J}(w)+\sum_{j=1}^m \alpha_j \int_0^{T_j}\left(b_j^2 w_j^2(t)+c_j^2 w_j^2\left( q_j\left(t\right) \right)\right) d t.
\label{eq_us7}
\end{equation}
According to Lemma 5 in \cite{bib5}, the following inequality holds:
\begin{equation}
\|w^\prime\|^2_{L_2\left(\mathcal{T}\right)} \geq C_3\|w\|^2_{W_2^1\left(\widetilde{\mathcal{T}}\right)} \quad \forall w\in \mathcal{W},
\label{eq_us8}
\end{equation}
where $C_3>0$, $w^\prime=\left[w_1^\prime,\dots,w_m^\prime\right]$. Combining (\ref{eq_us7}) with (\ref{eq_us4}) and (\ref{eq_us8}), we obtain the estimate
\begin{equation}
\frac{\alpha C_3}{3}\|w\|_{W_2^1\left(\widetilde{\mathcal{T}}\right)}^2 \leq \mathcal{J}(w)+K\|w\|_{L_2(\mathcal{T})}^2, \quad \alpha:=\min _{j=\overline{1, m}} \alpha_j>0.
\label{eq_us9}
\end{equation}

Further, by virtue of the compactness of the embedding operator from $W_2^1(\mathcal{T})$ into $L_2(\mathcal{T})$, there exists a sequence $\left\{w_{\left(n_k\right)}\right\}_{k \in \mathbb{N}}$ that converges in $L_2(\mathcal{T})$. Inequality (\ref{eq_us9}) gives
$$\frac{\alpha C_3}{3}\left\|w_{\left(n_k\right)}-w_{\left(n_l\right)}\right\|_{W_2^1\left(\widetilde{\mathcal{T}}\right)}^2 \leq \mathcal{J}\left(w_{\left(n_k\right)}-w_{\left(n_l\right)}\right)+K\left\|w_{\left(n_k\right)}-w_{\left(n_l\right)}\right\|_{L_2(\mathcal{T})}^2.$$
Moreover, using (\ref{eq_us2}) for $n=2$ along with (\ref{eq_us6}), we get
$$\mathcal{J}\left(w_{\left(n_k\right)}-w_{\left(n_l\right)}\right) \leq \frac{2}{n_k}+\frac{2}{n_l}.$$
Thus, $\left\{w_{\left(n_k\right)}\right\}_{k \in \mathbb{N}}$ is a Cauchy sequence also in $W_2^1\left(\widetilde{\mathcal{T}}\right)$. Let $w_{(0)}$ be its limit therein.

By virtue of Lemma \ref{L3}, the convergence of $w_{\left(n_k\right)}$ to $w_{(0)}$ in $W_2^1\left(\widetilde{\mathcal{T}}\right)$ implies $\ell w_{\left(n_k\right)} \rightarrow \ell w_{(0)}$ in $L_2(\mathcal{T})$. Then, due to (\ref{eq_us6}), we have
$$\left\|\ell w_{(0)}\right\|_{L_2(\mathcal{T})}^2=\lim _{k \rightarrow \infty}\left\|\ell w_{\left(n_k\right)}\right\|_{L_2(\mathcal{T})}^2 \leq \frac{1}{\alpha} \lim _{k \rightarrow \infty} \mathcal{J}\left(w_{\left(n_k\right)}\right)=0,$$
that is, $\ell w_{(0)}=0$. Thus, $w_{(0)}$ solves the Cauchy problem (\ref{eq_t2})--(\ref{eq_t5}) with $y_0=0$ and $u_j(t) \equiv 0$ for $j=\overline{1, m}$. Hence, we have $w_{(0)}=0$, which contradicts $\left\|w_{(0)}\right\|_{W_2^1(\mathcal{T})}=1$.
\end{proof}

We now proceed to prove the main result of this section.

\begin{theorem}\label{T5}
The boundary value problem $\mathcal{B}$ has a unique solution $y \in W_2^1\left(\mathcal{T}\right) \cap W_2^2\left(\widetilde{\mathcal{T}}\right)$. Moreover, there exists $C$ such that
\begin{equation}
\|y\|_{W_2^1\left(\mathcal{T}\right)} \leq C|y_0|.
\label{eq_us10}
\end{equation}
\end{theorem}
\begin{proof}
Consider the function $\Phi=\left[\Phi_1, \ldots, \Phi_m\right] \in W_2^1\left(\mathcal{T}\right)$ determined by the formulae
\[
\Phi_1\left(t\right) = \begin{cases} 
y_0\left(1-\frac{qt}{T_1}\right), & 0 \leqslant t < q^{-1}T_1, \\
0, & q^{-1}T_1 \leqslant t \leqslant T_1,
\end{cases}
\quad \Phi_j\left(t\right) \equiv 0, \quad j = 2, \dots, m.
\]
By virtue of Lemma \ref{L2}, for a function $y \in W_2^1\left(\mathcal{T}\right)$ obeying conditions (\ref{eq_t4})--(\ref{eq_t6}), to be a solution of the problem $\mathcal{B}$, it is necessary and sufficient to satisfy (\ref{eq_t8}). In other words, $y$ is a solution of $\mathcal{B}$ if and only if $x:=y-\Phi \in \mathcal{W}$ (which is equivalent to (\ref{eq_t4})--(\ref{eq_t6})) and
\begin{equation}
B(\Phi, w)+B(x, w)=0\quad\forall w \in \mathcal{W}
\label{eq_us11}
\end{equation}
(which, in turn, is equivalent to (\ref{eq_t8})).

Since $B(w, w)=\mathcal{J}(w)$, Lemmas  \ref{L3} and \ref{L4}  imply that $(\cdot, \cdot)_{\mathcal{W}}:=B(\cdot, \cdot)$  is an equivalent inner product on $\mathcal{W}$. Moreover, we have the estimate
\begin{equation}
\left|B\left(\Phi, w\right)\right|= \alpha_1\left|\displaystyle\int_{0}^{T_1}\ell_1\Phi\left(t\right)\ell_1 w\left(t\right) dt \right| \leqslant M|y_0| \| w\|_{\mathcal{W}},
\label{eq_us12}
\end{equation}
where $\|w\|_{\mathcal{W}}=\sqrt{(w, w)_{\mathcal{W}}}$. Thus, by virtue of the Riesz theorem on the general form of a linear bounded functional in a Hilbert space, there exists a unique $x \in \mathcal{W}$ such that (\ref{eq_us11}) is fulfilled. Hence, the problem $\mathcal{B}$ has the unique solution $y=\Phi+x$.

Finally, according to (\ref{eq_us11}) and (\ref{eq_us12}), we have
$$\|x\|_{\mathcal{W}}\leqslant M|y_0|.$$
which along with Lemma \ref{L4} gives
$$\|x\|_{W_2^1\left(\mathcal{T}\right)}=\|x\|_{W_2^1\left(\widetilde{\mathcal{T}}\right)}\leqslant\frac{M}{\sqrt{C_2}}|y_0|.$$
Using also the estimate
$$\|\Phi\|^2_{W^1_2\left(\mathcal{T}\right)} = \|\Phi_1\|^2_{W^1_2\left[0,q^{-1}T_1\right]} = \frac{T_1^2+3q^2}{3qT_1}y_0^2,$$
we arrive at (\ref{eq_us10}).
\end{proof}

\backmatter

\bmhead{Funding}
This work was supported by the Russian Science Foundation, project no. 24-71-10003.

\bmhead{Data availability} 
Data sharing is not applicable to this article.

\section*{Declarations}

Not applicable. 

\bmhead{Conflict of interest} 

The author has no relevant financial or non-financial interests to disclose.

\bmhead{Ethics approval and consent to participate}

Not applicable.

\bibliography{sn-bibliography}

%% BioMed_Central_Bib_Style_v1.01

\begin{thebibliography}{25}
% BibTex style file: bmc-mathphys.bst (version 2.1), 2014-07-24
\ifx \bisbn   \undefined \def \bisbn  #1{ISBN #1}\fi
\ifx \binits  \undefined \def \binits#1{#1}\fi
\ifx \bauthor  \undefined \def \bauthor#1{#1}\fi
\ifx \batitle  \undefined \def \batitle#1{#1}\fi
\ifx \bjtitle  \undefined \def \bjtitle#1{#1}\fi
\ifx \bvolume  \undefined \def \bvolume#1{\textbf{#1}}\fi
\ifx \byear  \undefined \def \byear#1{#1}\fi
\ifx \bissue  \undefined \def \bissue#1{#1}\fi
\ifx \bfpage  \undefined \def \bfpage#1{#1}\fi
\ifx \blpage  \undefined \def \blpage #1{#1}\fi
\ifx \burl  \undefined \def \burl#1{\textsf{#1}}\fi
\ifx \doiurl  \undefined \def \doiurl#1{\url{https://doi.org/#1}}\fi
\ifx \betal  \undefined \def \betal{\textit{et al.}}\fi
\ifx \binstitute  \undefined \def \binstitute#1{#1}\fi
\ifx \binstitutionaled  \undefined \def \binstitutionaled#1{#1}\fi
\ifx \bctitle  \undefined \def \bctitle#1{#1}\fi
\ifx \beditor  \undefined \def \beditor#1{#1}\fi
\ifx \bpublisher  \undefined \def \bpublisher#1{#1}\fi
\ifx \bbtitle  \undefined \def \bbtitle#1{#1}\fi
\ifx \bedition  \undefined \def \bedition#1{#1}\fi
\ifx \bseriesno  \undefined \def \bseriesno#1{#1}\fi
\ifx \blocation  \undefined \def \blocation#1{#1}\fi
\ifx \bsertitle  \undefined \def \bsertitle#1{#1}\fi
\ifx \bsnm \undefined \def \bsnm#1{#1}\fi
\ifx \bsuffix \undefined \def \bsuffix#1{#1}\fi
\ifx \bparticle \undefined \def \bparticle#1{#1}\fi
\ifx \barticle \undefined \def \barticle#1{#1}\fi
\bibcommenthead
\ifx \bconfdate \undefined \def \bconfdate #1{#1}\fi
\ifx \botherref \undefined \def \botherref #1{#1}\fi
\ifx \url \undefined \def \url#1{\textsf{#1}}\fi
\ifx \bchapter \undefined \def \bchapter#1{#1}\fi
\ifx \bbook \undefined \def \bbook#1{#1}\fi
\ifx \bcomment \undefined \def \bcomment#1{#1}\fi
\ifx \oauthor \undefined \def \oauthor#1{#1}\fi
\ifx \citeauthoryear \undefined \def \citeauthoryear#1{#1}\fi
\ifx \endbibitem  \undefined \def \endbibitem {}\fi
\ifx \bconflocation  \undefined \def \bconflocation#1{#1}\fi
\ifx \arxivurl  \undefined \def \arxivurl#1{\textsf{#1}}\fi
\csname PreBibitemsHook\endcsname

%%% 1
\bibitem[\protect\citeauthoryear{Krasovskii}{1968}]{bib1}
\begin{bbook}
\bauthor{\bsnm{Krasovskii}, \binits{N.N.}}:
\bbtitle{Theory of Motion Control}.
\bpublisher{Nauka},
\blocation{Moscow}
(\byear{1968})
\end{bbook}
\endbibitem

%%% 2
\bibitem[\protect\citeauthoryear{Skubachevskii}{1994}]{bib2}
\begin{barticle}
\bauthor{\bsnm{Skubachevskii}, \binits{A.L.}}:
\batitle{On the problem of damping a control system with aftereffect}.
\bjtitle{Dokl. Ross. Akad. Nauk}
\bvolume{335}(\bissue{2}),
\bfpage{157}--\blpage{160}
(\byear{1994})
\end{barticle}
\endbibitem

%%% 3
\bibitem[\protect\citeauthoryear{Skubachevskii}{1997}]{bib3}
\begin{bbook}
\bauthor{\bsnm{Skubachevskii}, \binits{A.L.}}:
\bbtitle{Elliptic Functional Differential Equations and Applications}.
\bpublisher{Birkhauser},
\blocation{Basel}
(\byear{1997})
\end{bbook}
\endbibitem

%%% 4
\bibitem[\protect\citeauthoryear{Adkhamova and Skubachevskii}{2024}]{bib4}
\begin{barticle}
\bauthor{\bsnm{Adkhamova}, \binits{A.S.}},
\bauthor{\bsnm{Skubachevskii}, \binits{A.L.}}:
\batitle{Damping problem for control system with delay with different number of
  inputs and outputs}.
\bjtitle{J. Math. Sci.}
\bvolume{286},
\bfpage{299}--\blpage{308}
(\byear{2024})
\end{barticle}
\endbibitem

%%% 5
\bibitem[\protect\citeauthoryear{Buterin}{2025}]{bib5}
\begin{barticle}
\bauthor{\bsnm{Buterin}, \binits{S.}}:
\batitle{On damping a control system with global aftereffect on quantum graphs:
  Stochastic interpretation}.
\bjtitle{Math. Meth. Appl. Sci.}
\bvolume{48},
\bfpage{4310}--\blpage{4331}
(\byear{2025})
\end{barticle}
\endbibitem

%%% 6
\bibitem[\protect\citeauthoryear{Buterin}{2024}]{bib6}
\begin{barticle}
\bauthor{\bsnm{Buterin}, \binits{S.A.}}:
\batitle{On damping a control system of arbitrary order with global aftereffect
  on a tree}.
\bjtitle{Math. Notes}
\bvolume{115},
\bfpage{877}--\blpage{896}
(\byear{2024})
\end{barticle}
\endbibitem

%%% 7
\bibitem[\protect\citeauthoryear{Buterin}{2025}]{bib7}
\begin{barticle}
\bauthor{\bsnm{Buterin}, \binits{S.A.}}:
\batitle{Control system on an infinite temporal tree}.
\bjtitle{Math. Notes}
\bvolume{117},
\bfpage{473}--\blpage{477}
(\byear{2025})
\end{barticle}
\endbibitem

%%% 8
\bibitem[\protect\citeauthoryear{Buterin}{2024}]{bib8}
\begin{bchapter}
\bauthor{\bsnm{Buterin}, \binits{S.A.}}:
\bctitle{Integro-differential control system on a temporal graph}.
In: \bbtitle{Proceedings of the International Conference ``Dynamical Systems:
  Stability, Control, Differential Games''},
pp. \bfpage{71}--\blpage{75}.
\bpublisher{IMM UrO RAN},
\blocation{Yekaterinburg}
(\byear{2024})
\end{bchapter}
\endbibitem

%%% 9
\bibitem[\protect\citeauthoryear{Montrol}{1970}]{bib9}
\begin{barticle}
\bauthor{\bsnm{Montrol}, \binits{E.}}:
\batitle{Quantum theory on a network}.
\bjtitle{J. Math. Phys.}
\bvolume{11}(\bissue{2}),
\bfpage{635}--\blpage{648}
(\byear{1970})
\end{barticle}
\endbibitem

%%% 10
\bibitem[\protect\citeauthoryear{Nicaise}{1987}]{bib10}
\begin{barticle}
\bauthor{\bsnm{Nicaise}, \binits{S.}}:
\batitle{Some results on spectral theory over networks, applied to nerve
  impulse transmission}.
\bjtitle{Math. Modelling}
\bvolume{9}(\bissue{6}),
\bfpage{437}--\blpage{449}
(\byear{1987})
\end{barticle}
\endbibitem

%%% 11
\bibitem[\protect\citeauthoryear{Langese et~al.}{1994}]{bib11}
\begin{bbook}
\bauthor{\bsnm{Langese}, \binits{J.}},
\bauthor{\bsnm{Leugering}, \binits{G.}},
\bauthor{\bsnm{Schmidt}, \binits{J.}}:
\bbtitle{Modelling, Analysis and Control of Dynamic Elastic Multi-Link
  Structures}.
\bpublisher{Birkhauser},
\blocation{Boston}
(\byear{1994})
\end{bbook}
\endbibitem

%%% 12
\bibitem[\protect\citeauthoryear{Pokornyi et~al.}{2005}]{bib12}
\begin{bbook}
\bauthor{\bsnm{Pokornyi}, \binits{V.V.}},
\bauthor{\bsnm{Penkin}, \binits{O.M.}},
\bauthor{\bsnm{Pryadiev}, \binits{V.L.}},
\bauthor{\bsnm{Borovskikh}, \binits{A.V.}},
\bauthor{\bsnm{Lazarev}, \binits{K.P.}},
\bauthor{\bsnm{Shabrov}, \binits{S.A.}}:
\bbtitle{Differential Equations on Geometric Graphs}.
\bpublisher{Fizmatlit},
\blocation{Moscow}
(\byear{2005})
\end{bbook}
\endbibitem

%%% 13
\bibitem[\protect\citeauthoryear{Berkolaiko and Kuchment}{2013}]{bib13}
\begin{bbook}
\bauthor{\bsnm{Berkolaiko}, \binits{G.}},
\bauthor{\bsnm{Kuchment}, \binits{P.}}:
\bbtitle{Introduction to Quantum Graphs}.
\bpublisher{AMS},
\blocation{Providence, RI}
(\byear{2013})
\end{bbook}
\endbibitem

%%% 14
\bibitem[\protect\citeauthoryear{Buterin}{2023}]{bib14}
\begin{botherref}
\oauthor{\bsnm{Buterin}, \binits{S.}}:
Functional-differential operators on geometrical graphs with global delay and
  inverse spectral problems.
Results Math.
\textbf{78}(79)
(2023)
\end{botherref}
\endbibitem

%%% 15
\bibitem[\protect\citeauthoryear{Wang and Yang}{2021}]{bib15}
\begin{barticle}
\bauthor{\bsnm{Wang}, \binits{F.}},
\bauthor{\bsnm{Yang}, \binits{C.-F.}}:
\batitle{Traces for sturm-liouville operators with constant delays on a star
  graph}.
\bjtitle{Results Math.}
\bvolume{76},
\bfpage{1}--\blpage{17}
(\byear{2021})
\end{barticle}
\endbibitem

%%% 16
\bibitem[\protect\citeauthoryear{Ambartsumian}{1944}]{bib16}
\begin{barticle}
\bauthor{\bsnm{Ambartsumian}, \binits{V.A.}}:
\batitle{On the theory of brightness fluctuations in the milky way}.
\bjtitle{Dokl. Akad. Nauk SSSR}
\bvolume{44},
\bfpage{244}--\blpage{247}
(\byear{1944})
\end{barticle}
\endbibitem

%%% 17
\bibitem[\protect\citeauthoryear{Hall and Wake}{1989}]{bib17}
\begin{barticle}
\bauthor{\bsnm{Hall}, \binits{A.}},
\bauthor{\bsnm{Wake}, \binits{G.}}:
\batitle{A functional differential equation arising in the modelling of cell
  growth}.
\bjtitle{J. Aust. Math. Soc. Ser.}
\bvolume{30},
\bfpage{424}--\blpage{435}
(\byear{1989})
\end{barticle}
\endbibitem

%%% 18
\bibitem[\protect\citeauthoryear{Kato and McLeod}{1971}]{bib18}
\begin{barticle}
\bauthor{\bsnm{Kato}, \binits{T.}},
\bauthor{\bsnm{McLeod}, \binits{J.}}:
\batitle{Functional-differential equation $\dot{y} = ay(\lambda t) + by(t)$}.
\bjtitle{Bull. Amer. Math. Soc.}
\bvolume{77}(\bissue{6}),
\bfpage{891}--\blpage{937}
(\byear{1971})
\end{barticle}
\endbibitem

%%% 19
\bibitem[\protect\citeauthoryear{Ockendon and Tayler}{1971}]{bib19}
\begin{barticle}
\bauthor{\bsnm{Ockendon}, \binits{J.}},
\bauthor{\bsnm{Tayler}, \binits{A.}}:
\batitle{The dynamics of a current collection system for an electric
  locomotive}.
\bjtitle{Proc. R. Soc. Lond. Ser. A Math. Phys. Eng. Sci.}
\bvolume{322},
\bfpage{447}--\blpage{468}
(\byear{1971})
\end{barticle}
\endbibitem

%%% 20
\bibitem[\protect\citeauthoryear{Iserles}{1993}]{bib20}
\begin{barticle}
\bauthor{\bsnm{Iserles}, \binits{A.}}:
\batitle{On the generalized pantograph functional-differential equation}.
\bjtitle{European J. Appl. Math.}
\bvolume{4},
\bfpage{1}--\blpage{38}
(\byear{1993})
\end{barticle}
\endbibitem

%%% 21
\bibitem[\protect\citeauthoryear{Iserles and Liu}{1997}]{bib21}
\begin{barticle}
\bauthor{\bsnm{Iserles}, \binits{A.}},
\bauthor{\bsnm{Liu}, \binits{Y.}}:
\batitle{On neutral functional-differential equations with proportional
  delays}.
\bjtitle{J. Math. Anal. Appl.}
\bvolume{207}(\bissue{1}),
\bfpage{73}--\blpage{95}
(\byear{1997})
\end{barticle}
\endbibitem

%%% 22
\bibitem[\protect\citeauthoryear{Djuric}{2025}]{bib26}
\begin{barticle}
\bauthor{\bsnm{Djuric}, \binits{N.}}:
\batitle{A new approach to inverse problems for
  \uppercase{S}turm-\uppercase{L}iouville operators with homogeneous delay}.
\bjtitle{Bol. Soc. Mat. Mex.}
\bvolume{31},
\bfpage{100}
(\byear{2025})
\end{barticle}
\endbibitem

%%% 23
\bibitem[\protect\citeauthoryear{Rossovskii}{2017}]{bib22}
\begin{barticle}
\bauthor{\bsnm{Rossovskii}, \binits{L.E.}}:
\batitle{Elliptic functional differential equations with contractions and
  extensions of independent variables of the unknown function}.
\bjtitle{J. Math. Sci.}
\bvolume{223}(\bissue{4}),
\bfpage{351}--\blpage{493}
(\byear{2017})
\end{barticle}
\endbibitem

%%% 24
\bibitem[\protect\citeauthoryear{Lednov}{2025}]{bib24}
\begin{bchapter}
\bauthor{\bsnm{Lednov}, \binits{A.P.}}:
\bctitle{On damping a neutral type control system on a temporal star graph with
  time-proportional delay}.
In: \bbtitle{Control Theory and Mathematical Modeling, Part 2},
pp. \bfpage{93}--\blpage{96}.
\bpublisher{Udmurt University},
\blocation{Izhevsk}
(\byear{2025})
\end{bchapter}
\endbibitem

%%% 25
\bibitem[\protect\citeauthoryear{Lednov}{2026}]{bib23}
\begin{botherref}
\oauthor{\bsnm{Lednov}, \binits{A.P.}}:
On damping a control system on a star graph with global time-proportional
  delay.
Ufa Math. J.,
1--11
(2026).
(In press, https://doi.org/10.48550/arXiv.2503.02522)
\end{botherref}
\endbibitem

\end{thebibliography}

\end{document}